\documentclass{amsart}
\usepackage{amsmath}
\usepackage{amsthm}
\usepackage{amssymb}
\usepackage{color}
\usepackage{graphicx}
\usepackage{apptools}
\usepackage{chngcntr}

\newtheorem{thm}{Theorem}\newtheorem{lemma}[thm]{Lemma}

\newtheorem{defi}[thm]{Definition}
\newtheorem{prop}[thm]{Proposition}

\AtAppendix{\counterwithin{thm}{section}}
\AtAppendix{\counterwithin{equation}{section}}

\newcommand{\rr}{{\mathbb{R}}}

\newcommand{\e}{\varepsilon}
\newcommand{\vip}{\vskip.2cm}
\newcommand{\indiq}{\hbox{\rm 1}{\hskip -2.8 pt}\hbox{\rm I}}
\newcommand{\dd}{{\rm d}}

\newcommand{\cM}{{\mathcal{M}}}

\newcommand{\cP}{{\mathcal{P}}}

\begin{document}

\title[Keller-Segel equation]
{A simple proof of non-explosion for measure solutions of the Keller-Segel equation}

\author{Nicolas Fournier and Yoan Tardy}
\address{Sorbonne Universit\'e, LPSM-UMR 8001, Case courrier 158, 75252 Paris Cedex 05, France.}
\email{nicolas.fournier@sorbonne-universite.fr, yoan.tardy@sorbonne-universite.fr}


\begin{abstract}
We give a simple proof, relying on a {\it two-particles} moment computation, that there exists a global
weak solution to the $2$-dimensional parabolic-elliptic Keller-Segel equation
when starting from any initial measure $f_0$ such that $f_0(\rr^2)< 8 \pi$.
\end{abstract}

\maketitle

\section{Introduction}

\subsection{The model}
We consider the classical parabolic-elliptic Keller-Segel model, also called Patlak-Keller-Segel, 
of chemotaxis in $\rr^2$, which writes
\begin{gather}\label{EDP1}
\partial_t f + \nabla \cdot ( f \nabla c) = \Delta f \qquad \hbox{and}\qquad  \Delta c + f =0.
\end{gather}
The unknown $(f,c)$ is composed of two nonnegative functions $f_t(x)$ and $c_t(x)$ 
of $t\geq 0$ and $x\in \rr^2$, and the initial condition $f_0$ is given.
\vip

This equation 
models the collective motion of a population of bacteria which emit a chemical substance that attracts them.
The quantity $f_t(x)$ represents the density of bacteria at position $x \in \rr^2$ at time $t\geq 0$,
while $c_t(x)$ represents the concentration of chemical substance at position $x \in \rr^2$ at time $t\geq 0$.
Note that in this model, the speed of diffusion of the chemo-attractant is supposed to be infinite.
This equation has been introduced by Keller and Segel \cite{ks}, see also Patlak \cite{p}.
We refer to the recent book of Biler \cite{b} and to the review paper of Arumugam and Tyagi \cite{at}
for some complete descriptions of what is known about this model.

\vip

We classically observe, see e.g. Blanchet-Dolbeault-Perthame 
\cite[Page 4]{bdp}, that necessarily
$ \nabla c_t = K\ast f_t $ for each $t\geq 0$, where 
\begin{align*}
K(x) = -\frac{x}{2\pi\|x\|^2} \quad\mbox{for } x \in \rr^2\setminus \{0\} \quad \mbox{ and (arbitrarily)} 
\quad K(0) = 0.
\end{align*}
Hence \eqref{EDP1} may be rewritten as
\begin{gather}\label{EDP}
\partial_t f + \nabla \cdot [ f\; (K\ast f)] = \Delta f.
\end{gather}

\subsection{Weak solutions}
We will deal with weak measure solutions. For each $M>0$, we set
$$
\cM_M(\rr^2) = \Big\{ \mu \mbox{ nonnegative measure on } \rr^2 \mbox{ such that } \mu (\rr^2) = M \Big\}
$$
and we endow $\cM_M(\rr^2)$ with the weak convergence topology,
i.e. taking $C_b(\rr^d)$, the set of continuous and bounded functions, as set of test functions.
We also denote by $C_b^2(\rr^d)$ the set of $C^2$-functions, bounded together with all their derivatives.
The following notion of weak solutions is classical, see e.g. Blanchet-Dolbeault-Perthame 
\cite[Page 5]{bdp}.

\begin{defi}\label{sf}
Fix $M>0$. We say that $f\in C([0,\infty), \cM_M(\rr^2))$ is a weak solution of \eqref{EDP} if 
for all $\varphi \in C^2_b(\rr^2)$, all $t\ge 0$, 
\begin{align*}
\int_{\rr^2} \varphi (x) f_t (\dd x) =& \int_{\rr^2} \varphi (x) f_0 (\dd x) + \int_0^t \int_{\rr^2} 
\Delta \varphi (x) f_s(\dd x)\dd s \\
&+  \frac12 \int_0^t \int_{\rr^2} \int_{\rr^2} K(x-y)\cdot ( \nabla \varphi (x) -  
\nabla \varphi (y))f_s(\dd x) f_s(\dd y) \dd s.
\end{align*}
\end{defi}

All the terms in this equality are well-defined. In particular concerning the last term,
it holds that $|K(x-y)\cdot ( \nabla \varphi (x) -  
\nabla \varphi (y))| \leq \|\nabla^2\varphi\|_\infty/2\pi$. However, 
$K(x-y)\cdot (\nabla \varphi (x)-\nabla \varphi (y))$, which equals $0$ when $x=y$ because we (arbitrarily)
imposed that $K(0)=0$, is not continuous near $x=y$. Hence a {\it good} weak solution
has to verify that $\int_{\rr^2} \int_{\rr^2} \indiq_{\{x=y\}}f_s(\dd x) f_s(\dd y)=0$ for a.e. $s\geq 0$.

\subsection{Main result}
Our goal is to give a simple prooof of the following global existence result.

\begin{thm}\label{thm}
Fix $M \in (0,8\pi)$ and assume that $f_0 \in \cM_M(\rr^2)$.
There exists a global weak solution $f$ to \eqref{EDP} with initial condition
$f_0$. Moreover, for all $\gamma \in (M/(4\pi),2)$, there is a constant $A_{M,\gamma}>0$
depending only on $M$ and $\gamma$ such that for all $T>0$,
\begin{equation}\label{ccs}
\int_0^T \int_{\rr^2} \int_{\rr^2} \|x-y\|^{\gamma-2} 
f_s(\dd x) f_s(\dd y)\dd s \leq A_{M,\gamma}(1+T).
\end{equation}
\end{thm}

These solutions indeed satisfy that 
$\int_{\rr^2} \int_{\rr^2} \indiq_{\{x=y\}}f_s(\dd x) f_s(\dd y)=0$ for a.e. $s\geq 0$.

\subsection{References}

Actually, a stronger result is already known: gathering the results of Bedrossian-Masmoudi \cite{bm} 
and Wei \cite{dw}, there exists a global {\it mild} solution for any $f_0\in \cM_M(\rr^2)$ with $M<8\pi$.
But the proof in \cite{bm,dw} is long and complicate, and the goal of the present
paper is to provide a simple and robust non explosion proof, even if the solution we build is weaker.
Actually, a global solution is also built in \cite{bm,dw} when $f_0\in \cM_{8\pi}(\rr^2)$ satisfies 
$\max_{x \in \rr^2}f_0(\{x\})<8\pi$, a case we could also treat with a little more work.

\vip
This model was first introduced by Patlak \cite{p} and Keller-Segel \cite{ks}, as a model for chemotaxis.
For an exhaustive summary of the knowledge about this equation and related models, we refer the reader 
to the review paper 
Arumugam-Tyagi \cite{at} and to the book of Biler \cite{b}.
The main difficulty of this model lies in the tight competition between diffusion and attraction. 
Therefore it is not clear that a solution exists because a blow-up could occur due to the emergence of a cluster,
i.e. a Dirac mass. Thus, the whole problem is about determining if the solution ends-up by being concentrated 
in finite time or not.

\vip
As shown in J\" ager-Luckhaus \cite{jl}, this depends on the initial mass of 
the solution $M = \int_{\rr^2}f_0(\dd x)$, 
the solution globally exists if $M$ is small enough and explodes in the other case. 
The fact that solutions must explode in finite time if $M>8\pi$ is rather easy to show.
But the fact that $8\pi$ is indeed the correct thresold was much more difficult.

\vip

Biler-Karch-Lauren\c cot-Nadzieja \cite{bklnD,bklnP} proved the global existence of a weak solution 
in the subcritical case 
for every initial data which is a radially symmetric measure such 
that $f_0(\{0\})=0$ and 
$f_0(\rr^2) = M \leq 8\pi$, with a few other anodyne technical conditions.
\vip

At the same time, Blanchet-Dolbeault-Perthame \cite{bdp} proved the existence of a global weak {\it free energy} solution 
for initial data $f_0\in L^1_+(\rr^2)$ with mass $M<8\pi$, a finite moment of order $2$ and a finite entropy.
The core of the argument lies in the use of the
logarithmic Hardy-Littlewood-Sobolev inequality applied on a well chosen free-energy quantity.
Something noticeable is that the authors use this inequality with its optimal constant to get the correct thresold $8\pi$.
\vip

In Bedrossian-Masmoudi \cite{bm}, it is proven that under the condition that $\max_{x\in \rr^2} f_0(\{x\})<8\pi$, 
one can build mild solutions even in the supercritical case, which are stronger solutions than weak solutions, 
but these solutions are local in time.
Wei \cite{dw} built global mild solutions in the subcritical and critical cases and local mild solutions in the 
supercritical case for every initial data $f_0\in L^1(\rr^2)$, without any other condition. And these
two last papers can be put together to build global mild solutions as soon as $f_0(\rr^2)\leq 8\pi$ and
$\max_{x\in \rr^2} f_0(\{x\})<8\pi$.

\vip

Let us finally mention \cite{fj}, where global weak solutions were built for any measure initial condition $f_0$
such that $f_0(\rr^2)<2\pi$, with a light additional moment condition. This work was 
inspired by the work of Osada \cite{o} on vortices. The present paper consists in refining this approach, 
and surprisingly,
this allows us to treat the whole subcritical case.

\subsection{Motivation}
Our main goal is to present a simple proof of non explosion.
This proof relies on a two-particles moment computation: roughly, we show that for 
$\gamma\in (0,2)$ and for $(f_t)_{t\geq 0}$ a solution to \eqref{EDP}, it {\it a priori} holds that
$$
\frac{\dd}{\dd t} \int_{\rr^2}\int_{\rr^2} \Big(||x-y||^\gamma \land 1\Big) f_t(\dd x)f_t(\dd y)
\geq c_{\gamma,M}\int_{\rr^2}\int_{\rr^2} ||x-y||^{\gamma-2}\indiq_{\{||x-y||\leq 1\}}f_t(\dd x)f_t(\dd y),
$$
with $c_{\gamma,M}>0$ as soon as $\gamma\in (4M/\pi,2)$. By integration, this implies \eqref{ccs} and such an
{\it a priori} estimate is sufficient to build a global solution.
\vip

This computation seems simple and robust. 
Although they build a more regular weak solution, Blanchet-Dolbeault-Perthame \cite{bdp}
use some optimal Hardy-Littlewood-Sobolev inequality. Moreover, they have some little restrictions
on the initial conditions (finite entropy and moment of order $2$).
The proof of Bedrossian-Masmoudi \cite{bm} and Wei \cite{dw} is much longer and relies on
a fine study of what happens near each possible atom of the initial condition. Let us say again
that they build a much stronger solution.

\vip

In particular, due to its robusteness, we hope to be able to apply such a method to study
the convergence of the empirical measure of some stochastic particle system, as the number of
particles tends to infinity, to the solution of \eqref{EDP}. 
To establish such a convergence, one needs to show
the  non explosion of the particle system, uniformly in $N$ in some sense. 
It seems that the present method works very well and we hope to be able to treat
the whole subcritical case $M\in(0,8\pi)$ and even the critical case $M=8\pi$.
To our knowledge, only the case where $M\in(0,4\pi)$ has been studied, by an entropy method, see
Bresch-Jabin-Wang \cite{bjw}.

\vip

We do not treat the criticial case $M=8\pi$ in the present paper for the sake of conciseness.

\section{Proof}

We fix $M>0$ and $f_0 \in \cM_M(\rr^2)$.
For $\e\in (0,1]$, we introduce the following regularized versions
$$
K_\e (x) = -\frac{x}{2\pi(\|x\|^2+\e)}\quad \mbox{and} \quad
f_0^\e(x ) =\frac{1}{2\pi \e} \int_{\rr^2} e^{-\|x-y\|^2/(2\e)}f_0(\dd y).
$$ 
of $K$ and $f_0$. Since $K_\e$ and $f_0^\e$ are smooth, the equation 
\begin{equation}\label{EDPesp}
\partial_t f^\e + \nabla \cdot [ f^\e (K_\e \ast f^\e)] = \Delta f^\e
\end{equation}
starting from $f^\e_0$ has a unique classical solution $(f^\e_t(x))_{t\geq 0,x\in \rr^2}$. This
solution preserves mass, i.e
\begin{align}\label{conservation}
\int_{\rr^2}f_t^\e(\dd x) = \int_{\rr^2}f_0^\e(\dd x) =\int_{\rr^2}f_0(\dd x)=M
\qquad \hbox{for all $t\geq 0$},
\end{align}
where we write $f^\e_t(\dd x) = f^\e_t(x) \dd x$.
Multiplying \eqref{EDPesp} by $\varphi\in C^2_b(\rr^2)$, integrating on $[0,t]\times\rr^2$,
proceeding to some integrations by parts and using a symmetry argument, we classically find that
\begin{align}\label{ff}
\int_{\rr^2} \varphi (x) f_t^\e(\dd x)  =& \int_{\rr^2} \varphi (x) f_0^\e(\dd x)  
+ \int_0^t \int_{\rr^2} \Delta \varphi (x) f_s^\e(\dd x)\dd s\\
&+ \frac12\int_0^t \int_{\rr^2}\int_{\rr^2}  K_\e(x-y)\cdot[\nabla \varphi(x)-\nabla\varphi(y)]
f_s^\e(\dd x) f_s^\e (\dd y) \dd s.\notag
\end{align}

We now prove some compactness result.

\begin{prop}\label{compacite}
Fix $M>0$, $f_0 \in \cM_M(\rr^2)$ and consider the corresponding family $(f^\e)_{\e \in (0,1]}$.
The family $(f^\e)_{\e \in (0,1]}$ is relatively compact in $C([0,\infty), \cM_M (\rr^2))$,
endowed with the uniform convergence on compact time intervals, $\cM_M(\rr^2)$ being endowed
with the weak convergence topology.
\end{prop}

\begin{proof} We first prove that for each $t\geq 0$, the family
$(f_t^\e)_{\e\in (0,1]}$ is tight in $\cP(\rr^2)$. 
Since the family $(f_0^\e)_{\e\in (0,1]}$ is clearly tight, 
by the de la Vall\'ee Poussin theorem,
there exists $\psi:\rr^2 \to \rr_+$ such that $\lim_{|x|\to \infty} \psi(x)=\infty$ and 
$A=\sup_{\e\in (0,1]}\int_{\rr^2}\psi(x)f_0^\e(\dd x)<\infty$. Moreover, we can choose $\psi$ smooth
and such that $||\nabla^2\psi||$ is bounded by some constant $C$. It then immediately follows from \eqref{ff},
since $\|z\|\, \|K_\e(z)\|\leq 1/(2\pi)$, that 
for all $\e \in (0,1]$, all $t\geq 0$,
$$
\int_{\rr^2} \psi(x) f_t^\e(\dd x) \leq \int_{\rr^2}\psi(x)f_0^\e(\dd x) + C\Big( M+ \frac{M^2}{4\pi}\Big)t
\leq A+C\Big( M+ \frac{M^2}{4\pi}\Big)t.
$$
As $\lim_{|x|\to \infty} \psi(x)=\infty$, we conclude that indeed, $(f_t^\e)_{\e\in (0,1]}$ is tight
for each $t\geq 0$.
\vip

By the Arzela-Ascoli theorem, it is enough to prove that
$f^\e$ is uniformly Lipschitz continuous in time, in that there 
exists a constant $C>0$ such that for all $\e\in (0,1]$, all $t\geq s \ge 0$, $\delta ( f_t^\e, f_s^\e) \le C|t-s|$,
where $\delta$ metrizes the weak convergence topology on $\cM_M(\rr^2)$. As is well-known,
we may find a family $(\varphi_n)_{n\ge 0}$ of elements of $C^2_b(\rr^2)$  satisfying
$$
\|\varphi_n\|_\infty + \|\nabla\varphi_n\|_\infty+ \|\nabla^2\varphi_n\|_\infty\le 1 \qquad \hbox{for all $n\geq 0$}
$$
and such that the distance $\delta$ on $\cM_M (\rr^2)$ defined through
$$
\delta (f,g) = \sum_{n\ge 0} 2^{-n} \Big|\int_{\rr^2} \varphi_n(x)f(\dd x) - \int_{\rr^2} \varphi_n(x)g(\dd x)\Big|
$$
is suitable. But using \eqref{ff}, for all $n\ge 0$, 
\begin{align*}
&\Big|\int_{\rr^2}\varphi_n(x) f_t^\e(\dd x) - \int_{\rr^2}\varphi_n(x) f_s^\e (\dd x) \Big|\\
=& \Big|\int_s^t \int_{\rr^2} \Delta \varphi_n(x) f_u^\e(\dd x)\dd u 
+ \frac12 \int_s^t \int_{\rr^2} \int_{\rr^2} K_\e(x-y)\cdot [\nabla \varphi_n (x)-\nabla \varphi_n(y)] 
f_u^\e(\dd x) f_u^\e (\dd y) \dd u \Big|\\
\le& ( M+M^2/(4\pi))(t-s),
\end{align*}
by \eqref{conservation}, since $\|\nabla^2\varphi_n\|_\infty\leq 1$ and since 
$\|z\|\, \|K_\e(z)\|\leq 1/(2\pi)$.
We conclude that 
$$
\delta (f_t^\e, f_s^\e) \le \sum_{n\ge 0} 2^{-n}(M+M^2/(4\pi))(t-s) = 2(M+M^2/(4\pi))(t-s)
$$
as desired.
\end{proof} 
 
The following simple geometrical observation is crucial for 
our purpose.

\begin{lemma}\label{inegalitebarycentre}
For all pair of nonincreasing functions $\varphi, \psi : (0,\infty) \to (0,\infty)$, for all $X,Y,Z \in \rr^2$ 
such that $X+Y+Z=0$, we have 
$$
\Delta = [ \varphi (\|X\|)X + \varphi (\|Y\|)Y + \varphi(\|Z\|)Z ] \cdot [\psi (\|X\|)X 
+ \psi (\|Y\|)Y + \psi(\|Z\|)Z ] \ge 0.
$$
\end{lemma}

\begin{proof}
We may study only the case where $\|X\| \le \|Y\| \le \|Z\|$. Since $Y=-X-Z$,
\begin{gather*}
\varphi (\|X\|)X + \varphi (\|Y\|)Y + \varphi(\|Z\|)Z = \lambda X - \mu Z, \\
\psi (\|X\|)X + \psi (\|Y\|)Y + \psi(\|Z\|)Z = \lambda' X - \mu' Z,
\end{gather*}
where $\lambda = \varphi (\|X\|)-\varphi (\|Y\|) \ge 0$, 
$\mu = \varphi (\|Y\|)-\varphi (\|Z\|) \ge 0$, 
$\lambda' = \psi (\|X\|)-\psi (\|Y\|) \ge 0$ and $\mu' = \psi (\|Y\|)-\psi (\|Z\|) \ge 0$. Therefore,
$$
\Delta = \lambda \lambda' \|X\|^2 + \mu \mu' \|Z\|^2 
- (\lambda \mu' + \lambda' \mu)X\cdot Z\geq 0
$$
as desired, because $X\cdot Z\leq 0$. Indeed,
if $X\cdot Z >0$, then
$\|Y\|^2=\|Z+X\|^2 = \|Z\|^2 + \|X\|^2 + 2X\cdot Z > \|Z\|^2\geq \|Y\|^2$, which is absurd. 
\end{proof}

The following computation is the core of the paper.

\begin{prop}\label{estimegamma}
Recall that $M \in (0,8\pi)$, that $f_0 \in \cM_M(\rr^2)$.
For all $\gamma \in (M/(4\pi), 2)$, there is a constant $A_{M,\gamma}>0$ depending only on $M$ and $\gamma$ 
such that for all $\e\in (0,1]$, all $T >0$,
$$
\int_0^T\int _{\rr^2} \int_ {\rr^2} \|x-y\|^{\gamma -2} f_s^\e(\dd x)f_s^\e(\dd y) \dd s 
\leq A_{M,\gamma}(1+T).
$$
\end{prop}

\begin{proof}
For any smooth $\psi:(\rr^2)^2\to\rr$ such that $\psi(x,y)=\psi(y,x)$, it holds that
\begin{align*}
\frac\dd{\dd t} \int_{\rr^2}\int_{\rr^2} \psi(x,y)f^\e_t(\dd x)f^\e_t(\dd y)
=&2\int_{\rr^2}\int_{\rr^2} \psi(x,y)[\Delta f^\e_t(x)-\nabla \cdot ( f^\e_t(x) (K_\e \ast f^\e_t))(x)]f^\e_t(y) 
\dd x \dd y\\
=&2\int_{\rr^2}\int_{\rr^2} [\Delta_x\psi(x,y) + (K_\e \ast f^\e_t)(x)\cdot\nabla_x\psi(x,y)]f^\e_t(x)f^\e_t(y) 
\dd x \dd y.
\end{align*}
We  fix $\gamma \in  (M/(4\pi), 2)$, introduce $\varphi(r)=r^{\gamma/2}/(1+r^{\gamma/2})$,
and set $\psi(x,y)=\varphi(||x-y||^2)$.
We have 
$$
\varphi'(r)=\frac\gamma 2 \frac{r^{\gamma/2-1}}{(1+r^{\gamma/2})^2}
\qquad \hbox{and}\qquad 
\varphi''(r) = \frac\gamma2 \frac{r^{\gamma/2-2}}{(1+r^{\gamma/2})^2}\Big(\frac\gamma2-1
-\frac{\gamma r^{\gamma/2}}{1+r^{\gamma/2}}\Big)
$$
and
\begin{gather*}
\nabla_x \psi(x,y)= 2 \varphi'(||x-y||^2)(x-y)= \gamma \frac{||x-y||^{\gamma-2}}{(1+||x-y||^{\gamma})^2}(x-y),\\
\Delta_x \psi(x,y)=4 \varphi'(||x-y||^2)+ 4||x-y||^2\varphi''(||x-y||^2)
= \gamma^2 \frac{||x-y||^{\gamma-2}}{(1+||x-y||^{\gamma})^2}\Big(1
-2 \frac{||x-y||^{\gamma}}{1+||x-y||^{\gamma}}\Big).
\end{gather*}
Hence
\begin{equation}\label{x1}
\frac\dd{\dd t} \int_{\rr^2}\int_{\rr^2} \varphi(||x-y||^2)f^\e_t(\dd x)f^\e_t(\dd y)= J_t^\e+S_t^\e,
\end{equation}
where
\begin{align*}
J^\e_t=&2\gamma^2 \int_{\rr^2}\int_{\rr^2} \frac{||x-y||^{\gamma-2}}{(1+||x-y||^{\gamma})^2}\Big(1
-2 \frac{||x-y||^{\gamma}}{1+||x-y||^{\gamma}}\Big) f^\e_t(x)f^\e_t(y) \dd x \dd y,\\
S^\e_t=&2\gamma \int_{\rr^2}\int_{\rr^2}\int_{\rr^2} \frac{||x-y||^{\gamma-2}}{(1+||x-y||^{\gamma})^2}(x-y)\cdot
K_\e(x-z) f^\e_t(x)f^\e_t(y)f^\e_t(z) \dd x \dd y \dd z.
\end{align*}

First, we have
\begin{equation}\label{x2}
J^\e_t \geq \gamma (\gamma+M/(4\pi)) \int_{\rr^2}\int_{\rr^2} 
\frac{||x-y||^{\gamma-2}}{(1+||x-y||^{\gamma})^2}
f^\e_t(x)f^\e_t(y) \dd x \dd y - M^2C_{M,\gamma},
\end{equation}
where $C_{M,\gamma}>0$ is a constant such that for all $a>0$, recall that $\gamma>M/(4\pi)$,
$$
2\gamma^2 \frac{a^{\gamma-2}}{(1+a^{\gamma})^2}\Big(1-2 \frac{a^{\gamma}}{1+a^{\gamma}}\Big) 
\geq 2\gamma \Big(\frac{\gamma+M/(4\pi)}2\Big) \frac{a^{\gamma-2}}{(1+a^{\gamma})^2} - C_{M,\gamma}.
$$

Next, by symmetrization, we have
\begin{align}\label{jj}
S^\e_t=&\gamma \int_{\rr^2}\int_{\rr^2}\int_{\rr^2} \frac{||x-y||^{\gamma-2}}{(1+||x-y||^{\gamma})^2}(x-y)\cdot
(K_\e(x-z)-K_\e(y-z)) f^\e_t(x)f^\e_t(y)f^\e_t(z) \dd x \dd y \dd z \\
=& \frac\gamma3 \int_{\rr^2}\int_{\rr^2}\int_{\rr^2} F_\e(x,y,z)f^\e_t(x)f^\e_t(y)f^\e_t(z) \dd x \dd y \dd z, \notag
\end{align}
where
\begin{align*}
F_\e(x,y,z) =& [K_\e(x-z)-K_\e(y-z)]\cdot (x-y) \frac{||x-y||^{\gamma-2}}{(1+||x-y||^{\gamma})^2} \\
&+[K_\e(y-x)-K_\e(z-x)]\cdot (y-z) \frac{||y-z||^{\gamma-2}}{(1+||y-z||^{\gamma})^2}\\
&+[K_\e(z-y)-K_\e(x-y)]\cdot (z-x)\frac{||z-x||^{\gamma-2}}{(1+||z-x||^{\gamma})^2}.
\end{align*}
Introducing $X=x-y$, $Y=y-z$ and $Z=z-x$ and recalling that $K_\e(X)=\frac{-X}{2\pi(\|X\|^2+\e)}$, we find
\begin{align*}
2\pi F_\e(x,y,z)=&\Big[\frac{Z}{\|Z\|^2+\e}+\frac{Y}{\|Y\|^2+\e}\Big]\cdot X 
\frac{||X||^{\gamma-2}}{(1+||X||^{\gamma})^2}\\
&+\Big[\frac{X}{\|X\|^2+\e}+\frac{Z}{\|Z\|^2+\e}\Big]\cdot Y\frac{||Y||^{\gamma-2}}{(1+||Y||^{\gamma})^2}\\
&+\Big[\frac{Y}{\|Y\|^2+\e}+\frac{X}{\|X\|^2+\e}\Big]\cdot Z\frac{||Z||^{\gamma-2}}{(1+||Z||^{\gamma})^2}.
\end{align*}
We now introduce 
\begin{align*}
G(x,y,z)=& \frac{\|X\|^{\gamma-2}}{(1+\|X\|^\gamma)^{2}}+\frac{\|Y\|^{\gamma-2}}{(1+\|Y\|^\gamma)^{2}}
+\frac{\|Z\|^{\gamma-2}}{(1+\|Z\|^\gamma)^{2}}\\
\geq &X\cdot\frac{X}{\|X\|^2+\e}\frac{\|X\|^{\gamma-2}}{(1+\|X\|^\gamma)^{2}}
+Y\cdot\frac{Y}{\|Y\|^2+\e}\frac{\|Y\|^{\gamma-2}}{(1+\|Y\|^\gamma)^{2}}
+Z\cdot\frac{Z}{\|Z\|^2+\e}\frac{\|Z\|^{\gamma-2}}{(1+\|Z\|^\gamma)^{2}}.
\end{align*}
Hence $G(x,y,z)+2\pi F_\e(x,y,z)$ is larger than
\begin{align*}
\Big(\frac{X}{\|X\|^2\!+\!\e}\!+\!\frac{Y}{\|Y\|^2\!+\!\e}\!+\!\frac{Z}{\|Z\|^2\!+\!\e}&\Big)\cdot
\Big(X\frac{||X||^{\gamma-2}}{(1+||X||^{\gamma})^2}+Y\frac{||Y||^{\gamma-2}}{(1+||Y||^{\gamma})^2}+
Z\frac{||Z||^{\gamma-2}}{(1+||Z||^{\gamma})^2}\Big),
\end{align*}
which is nonnegative according to Lemma \ref{inegalitebarycentre}, since $r\to 1/(r^2+\e)$
and $r\to r^{\gamma-2}/(1+r^\gamma)^{2}$ are both nonincreasing on $(0,\infty)$
and since $X+Y+Z=0$. Thus $F_\e(x,y,z)\geq -G(x,y,z)/(2\pi)$. Recalling \eqref{jj}, we conclude that
\begin{align}\label{x3}
S^\e_t \ge& -\frac\gamma{6\pi} \int_{\rr^2} \!\int_{\rr^2}\!\int_{\rr^2}\! \!
G(x,y,z) f_t^\e (\dd x) f_t^\e(\dd y) f_t^\e (\dd z)\\
=& - \frac{\gamma M}{2\pi}  \int_{\rr^2} \int_{\rr^2}  \frac{\|x-y\|^{\gamma-2}}{(1+\|x-y\|^\gamma)^2} 
f_t^\e (\dd x) f_t^\e(\dd y)\notag
\end{align}
by symmetry again.
\vip
Gathering \eqref{x1}-\eqref{x2}-\eqref{x3}, we find
$$
\frac\dd{\dd t} \int_{\rr^2}\int_{\rr^2} \varphi(||x-y||^2)f^\e_t(\dd x)f^\e_t(\dd y)
\geq \gamma \Big(\gamma-\frac{M}{4\pi}\Big)\int_{\rr^2} \int_{\rr^2} 
\frac{\|x-y\|^{\gamma-2}}{(1+\|x-y\|^\gamma)^2} f_t^\e (\dd x) f_t^\e(\dd y)-M^2C_{M,\gamma}.
$$
Integrating on $[0,T]$, using that $\gamma>M/(4\pi)$ and and that $\varphi$ is $[0,1]$-valued,
we end with
$$
\int_0^T\int_{\rr^2} \int_{\rr^2} \frac{\|x-y\|^{\gamma-2}}{(1+\|x-y\|^\gamma)^2} f_t^\e (\dd x) f_t^\e(\dd y)\dd t
\leq \frac{M^2 + M^2C_{M,\gamma}T}{\gamma(\gamma-M/(4\pi))}.
$$
One easily completes the proof, using that there is $D_\gamma>0$ such that 
$a^{\gamma-2} \leq 2 a^{\gamma-2}/(1+a^\gamma)^2+D_\gamma$
for all $a>0$
\end{proof}

We finally give the 

\begin{proof}[Proof of Theorem \ref{thm}]
Recall that $M\in (0,8\pi)$, that $f_0 \in \cM_M(\rr^2)$, and that $(f^\e)_{\e\in (0,1]}$ is the corresponding 
family of regularized solutions.
By Proposition \ref{compacite}, we can find $(\e_k)_{k\ge 0}$ and 
$f\in C([0,\infty),\cM_M(\rr^2))$ such that $\lim_k\e_k=0$ and
$\lim_k f^{\e_k}= f$ in $C([0,\infty),\cM_M(\rr^2))$,
endowed with the uniform convergence on compact time intervals, $\cM_M(\rr^2)$ being endowed with the
weak convergence topology. By definition of $f_0^\e$, we obviously have $f|_{t=0}=f_0$.
By Proposition \ref{estimegamma} and the Fatou lemma, for all $\gamma \in (M/(4\pi),2)$, we have
\begin{equation}\label{ess}
\int_0^T \int_{\rr^2} \int_{\rr^2} \|x-y\|^{\gamma-2}
f_s(\dd x) f_s(\dd y) \dd s \leq A_{M,\gamma}(1+T).
\end{equation}

It only remains to check that $f$ is a weak solution to \eqref{EDP}. 
We fix $\varphi \in C^2_b(\rr^2)$
and use \eqref{ff} to write
$\int_{\rr^2} \varphi (x) f_t^{\e_k}(\dd x) = I_k(t)+J_k(t)$, where  
\begin{gather*}
I_k(t)=\int_{\rr^2} \varphi (x) f_0^{\e_k}(\dd x)+\int_0^t \int_{\rr^2} \Delta \varphi (x) f_s^{\e_k}(\dd x)\dd s,\\
J_k(t)=\frac12\int_0^t \int_{\rr^2}\int_{\rr^2}  K_{\e_k}(x-y)\cdot [\nabla\varphi (x)-\nabla \varphi(y)] 
f_s^{\e_k}(\dd x) f_s^{\e_k} (\dd y) \dd s.
\end{gather*}
Since $\varphi$ and $\Delta \varphi$ are continuous and bounded, we immediately conclude that
$$
\lim_k \int_{\rr^2} \!\varphi (x) f_t^{\e_k}(\dd x)=\int_{\rr^2} \!\varphi (x) f_t(\dd x)
\quad \hbox{and}\quad 
\lim_k I_k(t)=\int_{\rr^2} \!\varphi (x) f_0(\dd x)+ \int_0^t \!\int_{\rr^2} \!\Delta \varphi (x) f_s(\dd x)\dd s,
$$
and it only remains to check that for
$J(t)=\frac12\int_0^t \int_{\rr^2}\int_{\rr^2}  K(x-y)\cdot[\nabla \varphi (x)-\nabla \varphi(y)] 
f_s(\dd x) f_s (\dd y) \dd s$, we have $\lim_k J_k(t)=J(t)$. To this end, we write $J_k(t)=J_k^1(t)+J_k^2(t)$, 
where
\begin{gather*}
J_k^1(t)=\frac12\int_0^t \int_{\rr^2}\int_{\rr^2}  K(x-y)\cdot [\nabla\varphi (x)-\nabla \varphi(y)] 
f_s^{\e_k}(\dd x) f_s^{\e_k} (\dd y) \dd s,\\
J_k^2(t)=\frac12\int_0^t \int_{\rr^2}\int_{\rr^2}  [K_{\e_k}(x-y)-K(x-y)]\cdot [\nabla\varphi (x)-\nabla \varphi(y)] 
f_s^{\e_k}(\dd x) f_s^{\e_k} (\dd y) \dd s.
\end{gather*}
Recalling the expression of $K$ and that $\varphi \in C^2_b(\rr^2)$, we see that 
$g(x,y)=K(x-y)\cdot [\nabla \varphi (x)-\nabla \varphi(y)]$ is bounded and continuous on the set
$\rr^2 \setminus D$, where $D=\{(x,y) \in \rr^2\times \rr^2:x=y\}$.
Since $f^{\e_k}_s\otimes f^{\e_k}_s$ goes weakly to $f_s\otimes f_s$ for each $s\geq 0$
and since $(f_s\otimes f_s)(D)=0$ for a.e. $s\geq 0$ by \eqref{ess}, 1, we conclude that
$\lim_k \int_{\rr^2}\int_{\rr^2} g(x,y)f_s^{\e_k}(\dd x) f_s^{\e_k} (\dd y)= \int_{\rr^2}\int_{\rr^2} g(x,y)f_s(\dd x) 
f_s(\dd y)$ for a.e. $s\geq 0$, whence $\lim_k J^1_k(t)=J_k(t)$ by dominated convergence.

\vip

We finally have to verify that $\lim_k J^2_k(t)=0$. We fix $\gamma \in (M/(4\pi),2)$ and write
$$
\|z\|\,\|K(z)-K_\e(z)\|
=\frac{\e}{2\pi(\e+\|z\|^2)}
\leq \min(1,\e\|z\|^{-2})\leq (\e\|z\|^{-2})^{1-\gamma/2}= \e^{1-\gamma/2}\|z\|^{\gamma-2}.
$$
Thus
\begin{align*}
|J^2_k(t)| \leq& \|\nabla^2 \varphi\|_\infty \e_k^{1-\gamma/2} \int_0^t \int_{\rr^2}\int_{\rr^2} \|x-y\|^{\gamma-2}
f_s^{\e_k}(\dd x) f_s^{\e_k} (\dd y) \dd s,
\end{align*}
which tends to $0$ as desired since $\sup_{\e\in (0,1]}\int_0^t \int_{\rr^2}\int_{\rr^2} \|x-y\|^{\gamma-2}
f_s^{\e_k}(\dd x) f_s^{\e_k} (\dd y) \dd s<\infty$
by Proposition \ref{estimegamma}.
\end{proof}

\end{document}